\documentclass[a4paper]{amsart}
\usepackage{amsgen, amstext,amsbsy,amsopn, amsthm, amsfonts,amssymb,amscd,amsmath,euscript,enumerate,url,verbatim,calc,xypic,tikz}
\usepackage{amsmath}
\usepackage{enumitem}
 \usepackage{xcolor}
 \usepackage{tikz-cd}
\usepackage{graphicx}

\usepackage{mathrsfs}  
\usepackage{hyperref}
\usepackage{MnSymbol}
\usepackage{pgfkeys}
\usepackage{mathtools}
 \usepackage{MnSymbol}
\usepackage{pgfkeys}
\usepackage{mathtools}
\def\multiset#1#2{\ensuremath{\left(\kern-.3em\left(\genfrac{}{}{0pt}{}{#1}{#2}\right)\kern-.3em\right)}}

\usetikzlibrary{arrows}

\usepackage{latexsym}
\usepackage{graphics}
\usepackage{color}
\usepackage{lastpage}
\usepackage{fancyhdr}
\usepackage{multirow}
\allowdisplaybreaks
\usepackage{graphicx}
\graphicspath{ {F:/IMAGES/} }

\newcommand{\m}{\mathfrak{m} }

\newcommand{\wt}{\widetilde}

\newcommand{\R}{\mathcal{R}}

\newcommand{\I}{\mathcal{I}}

\newcommand{\reg}{\operatorname{reg}}

\newcommand{\depth}{\operatorname{depth}}

\newcommand{\iso}{\cong} 
\newcommand{\proset}{\,\mathrel{\lower 4pt\hbox{$\scriptscriptstyle/$}
		\mkern -14mu\subseteq }\,} 
		\usepackage{stackengine}
\newcommand\tsup[2][2]{%
 \def\useanchorwidth{T}%
  \ifnum#1>1%
    \stackon[-.5pt]{\tsup[\numexpr#1-1\relax]{#2}}{\scriptscriptstyle\sim}%
  \else%
    \stackon[.5pt]{#2}{\scriptscriptstyle\sim}%
  \fi%
}
\newtheorem{theorem}{Theorem}[section]
\newtheorem{corollary}[theorem]{Corollary}
\newtheorem{lemma}[theorem]{Lemma}

\usepackage{amsmath}

\theoremstyle{definition}

\newtheorem{remark}[theorem]{Remark}
\newtheorem{definition}[theorem]{Definition}

\newtheorem{example}[theorem]{Example}

\title[Bounding reduction number and the Hilbert coefficients]{Bounding reduction number and the Hilbert coefficients of filtration}

\author[K Saloni and A K Yadav]{Kumari Saloni and Anoot Kumar Yadav}
\thanks{The first author is supported by SERB Power grant No.: $\mbox{SPG}/2022/002099$, Govt. of India}  
\thanks{ The second author is supported by a UGC fellowship NTA Ref. no. 191620042352, Govt. of India.}
\address{Department of Mathematics, Indian Institute of Technology Patna, Bihta, Patna 801106, India}
\email{ksaloni@iitp.ac.in}
\address{Department of Mathematics, Indian Institute of Technology Patna, Bihta, Patna 801106, India} 
\email{anoot\_2021ma06@iitp.ac.in,\; vicky.anoot@gmail.com}
\begin{document}
\begin{abstract}
    Let $(A,\m)$ be a Cohen-Macaulay local ring of dimension $d\geq 3$, $I$ an $\m$-primary ideal and $\mathcal{I}=\{I_n\}_{n\geq 0}$ an $I$-admissible filtration. We establish bounds for the third Hilbert coefficient: (i) $e_3(\mathcal{I})\leq e_2(\mathcal{I})(e_2(\mathcal{I})-1)$ and (ii)  $e_3(I)\leq e_2(I)(e_2(I)-e_1(I)+e_0(I)-\ell(A/I))$ if $I$ is an integrally closed ideal. Further, assume the respective boundary cases along with the vanishing of $e_i(\mathcal{I})$ for $4\leq i\leq d$. Then we show that the associated graded ring of the Ratliff-Rush filtration of $\mathcal{I}$ is almost Cohen-Macaulay, Rossi's bound for the reduction number $r_J(I)$ of $I$ holds true and the reduction number of Ratliff-Rush filtration of $\mathcal{I}$ is bounded above by $r_J(\I).$ In addition, if $\wt{I^{r_J(I)}}=I^{r_J(I)}$, then we prove that 
    $\reg G_I(A)=r_J(I)$ 
     and a bound on the stability index of Ratliff-Rush filtration is obtained. We also do a parallel discussion on the \textquotedblleft good behaviour of the Ratliff-Rush filtration with respect to superficial sequence''. 
    \end{abstract}
\maketitle
\section{Introduction}
Throughout the paper, let $(A,\m)$ be a Cohen-Macaulay local ring of dimension $d\geq 1$ with infinite residue field and $I$ an $\m$-primary ideal. A celebrated result of Rossi \cite{r} gives a linear bound for the reduction number of $I$ in terms of multiplicity and the first Hilbert coefficient of $I$ when $d=2$, see \eqref{rossi bound}. It is a conjecture to establish the same bound in higher dimensions which is the primary motivation of this paper. In \cite{ms} and \cite{sy},  quadratic bounds, involving the second and the third Hilbert coefficients of $I$ were given in dimension three . Authors  in \cite{sy} observed that the \textquotedblleft good behaviour of the Ratliff-Rush filtration'' plays a crucial role in this context. This paper further builds upon the ideas of \cite{sy} in order to obtain computable conditions for the above observation and consequently makes progress in the direction of Rossi's conjecture. 

Let $\mathcal{I}$ be an $I$-admissible filtration. Recall that 
an $I$-admissible filtration  $\mathcal I=\{I_n\}_{n\in \mathbb Z}$ is a sequence of ideals  such that $(i)$ $I_{n+1}\subseteq I_n$,
   $(ii)$  $I_m I_n\subseteq I_{m+n}$ and $(iii)~ I^n\subseteq I_n\subseteq I^{n-k}$ for some $k\in \mathbb N$. The Hilbert coefficients of $\mathcal{I}$ are the unique integers $e_i(\mathcal{I})$, $0\leq i\leq d,$ such that the function $H_{\mathcal I}(n):=\ell{(R/I_n)}$ coincides with the following polynomial for $n\gg0:$  $$P_{\mathcal I}(x) =e_0(\mathcal I){x+d-1\choose d}-e_1(\mathcal{ I}){x+d-2\choose d-1}+\cdots +(-1)^de_d(\mathcal I).$$ 
Here, $\ell(*)$ denotes the length function. The function $H_{\mathcal I}(n)$ and the polynomial $P_{\mathcal I}(x)$ are known as the Hilbert-Samuel function and the Hilbert-Samuel polynomial of $\mathcal I$, respectively. We refer to \cite{rv} for details.  For $\mathcal I=\{I^n\}_{n\geq 0}$, we write $e_i(I)$ instead of $e_i(\mathcal I)$.  A {\it reduction} \index{reduction} of $\mathcal I$ is an ideal $J\subseteq I_1$ such that $JI_n=I_{n+1}$ for  $n\gg 0$ and it is called a {\it minimal reduction} if it is minimal with respect to containment among all reductions. If the residue field $R/\m$ is infinite, then a minimal reduction of $\mathcal I$ exists and is generated by $d$ elements. Minimal reductions are important in the study of Hilbert functions and blow-up algebras. For a minimal reduction $J$ of $\mathcal I$, we define 
    \[r_J(\mathcal I)= \sup \{n\in \mathbb Z \mid I_n\not= JI_{n-1}\} \mbox{ and } 
  r(\mathcal I)=\min\{r_J(\mathcal I)~|~ J \mbox{ is a minimal reduction of } \mathcal I\}.\]
 The above numbers are known as the {\it{reduction number of $\mathcal I$ with respect to $J$}} and  the {\it{reduction number of $\mathcal I$}}  respectively. When $\mathcal I$ is the $I$-adic filtration $\{I^n\}_{n\in\mathbb{Z}}$,  we write $r_J(I)$ and $r(I)$ in place of $r_J(\mathcal I)$ and $r(\mathcal I)$ respectively. If $J$ is a reduction of $I$, then $R[It]$ is module-finite over $R[Jt]$ and the  reduction number is the largest degree of the elements in a minimal homogeneous generating set of the ring $R[It]$ over $R[Jt].$ 
 Suppose $d\leq 2$. Rossi \cite[Corollary 1.5]{r} proved that, for a minimal reduction $J\subseteq I$, 
 \begin{equation}\label{rossi bound}
     r_J(I)\leq e_1(I)-e_0(I)+\ell(A/I)+1.
 \end{equation}
Establishing the same bound in  higher dimensions is an open problem.  Many attempts have been made to achieve bounds of similar character. For instance, the bound in \eqref{rossi bound} holds in all dimensions if $\depth G_I(A) \geq d-2$ \cite[Theorem 4.3]{rv} or if $e_1(I)-e_0(I)+\ell (A/I)=1$ \cite[Theorem 3.1]{gno}. Another case is when $I\subseteq k[x,y,z]$ is of codimension $3$ generated by five quadrics \cite[Theorem 2.1 and Proposition 2.4]{hsv}. We prove the following theorem. 
\begin{theorem}\label{red-bd}
  Suppose $d\geq 3.$ If the Ratliff-Rush filtration of $I$ behaves well modulo a superficial sequence $x_1,\ldots,x_{d-2}$, then  $r_J(I)\leq e_1(I)-e_0(I)+\ell(A/I)+1$.  
\end{theorem}
 We recall that the Ratliff-Rush filtration  of $\mathcal{I}$ is the filtration $\wt{\mathcal I}=\{\widetilde{I_n}=\mathop\bigcup\limits_{t\geq 0}(I_{n+t}:I^t)\}_{n\in\mathbb{Z}}$. It is well known  that $\wt{\mathcal{I}}$ is an $I$-admissible filtration, see \cite{bl}. 
Let $x_1,\ldots,x_s\in I$ be a superficial sequence for $\mathcal{I}$ and $A^\prime=A/(x_1)$. We say that the {\it{Ratliff-Rush filtration of $I$ behaves well modulo  $x_1$}} if $\widetilde{I^n}A^\prime = \widetilde{I^nA^\prime}$
 for all $n\geq 0.$  This definition can be  extended inductively, see Definition \ref{def1}. This notion, introduced in \cite{Put2} by Puthenpurakal, has been the main tool for solving Itoh's conjecture in the work of Puthepnurakal,  \cite{Put3} and \cite{Put4}.  In fact,  many results of Marley's thesis  can be obtained by replacing the condition $\depth G_{\mathcal{I}}(A)\geq d-1$ with the good behaviour of the Ratliff-Rush filtration in the above sense, see Theorem \ref{conseq of behav of RREF}. We observe that the Ratliff-Rush filtration of $I$ behaves well mod a superficial sequence $x_1,\ldots,x_{d-2}$ is weaker than  $\depth G_I(A)\geq d-1$ as evident by Example \ref{ex1.2}. In Theorem \ref{rrf}, we show that for an integrally closed ideal $I$, if $e_2(I)=e_1(I)-e_0(I)+\ell(A/I)$ and $e_i(I)=0$ for $3\leq i \leq d$  then the Ratliff-Rush filtration of $I$ behaves well modulo a superficial sequence of length $d-1.$ This result is analogous to \cite[Theorem  6.2]{Put2}.

Now suppose $d\geq 3$ and $\depth G(I)\geq d-3.$ In \cite[Theorem 4.1]{ms}, it was proved that for an $\m$-primary ideal $I$, 
\begin{equation}\label{ms bound}
    r_J(I)\leq e_1(I)-e_0(I)+\ell(A/I)+1+(e_2(I)-1)e_2(I)-e_3(I).
\end{equation}
Further, in \cite[Theorem 5.6]{sy}, the authors proved that for an an $\m$-primary integrally closed ideal $I$, 
\begin{equation}\label{sy bound}
r_J(I)\leq e_1(I)-e_0(I)+\ell(A/I)+1+e_2(I)(e_2(I)-e_1(I)+e_0(I)-\ell(A/I))-e_3(I).
\end{equation}
We refer to the value $(e_2(I)-1)e_2(I)-e_3(I)$ or $e_2(I)(e_2(I)-e_1(I)+e_0(I)-\ell(A/I))-e_3(I)$ for integrally closed $I$, as the error term. Clearly, when the error term is zero, Rossi's bound holds in dimension three. We prove the following theorem as an extension of the above observation.  
 \begin{theorem}\label{theorem3}
Let $(A, \mathfrak{m})$ be a Cohen-Macaulay local ring of dimension $d \geq 3$, $I$ an $\mathfrak{m}$-primary ideal. If either of the following two conditions holds:
 \renewcommand{\labelenumi}{(\roman{enumi})}
\begin{enumerate}
    \item $e_3(I) = e_2(I)(e_2(I) - 1)$,
    \item $I$ is integrally closed and $e_3(I) = e_2(I)(e_2(I) - e_1(I) + e_0(I) + \ell(A/I))$,
\end{enumerate}
and $e_k(I) = 0$ for $4 \leq k \leq d$. Then, $$r_J(I)\leq e_1(I)-e_0(I)+\ell(A/I)+1.$$

 \end{theorem}
The above theorem indicates that in higher dimensions, aiming to find a bound which also involves higher Hilbert coefficients could be a useful approach for solving the conjecture on Rossi's bound. In the further theorems, we show that the error terms of \eqref{ms bound} and \eqref{sy bound} are indeed non-negative and the vanishing of error terms along with the $e_i(I)$ for $i\geq 4$ has interesting consequences. 
For an $I$-admissible filtration $\mathcal{I}$, we write $G_{\wt{\mathcal{I}}}(A)=\displaystyle{\bigoplus_{n\geq 0}}\frac{\wt{I_n}}{\wt{I_{n+1}}}$ for the associated graded ring of
$\wt{\mathcal{I}}.$

 \begin{theorem}\label{theorem1}
 Let $(A,\m)$ be a Cohen-Macaulay local ring of dimension $d\geq 3$ and $I$ an $\m$- primary ideal. Suppose $\mathcal{I}=\{I_n\}_{n\geq 0}$ is an $I$-admissible filtration. Then 
    \begin{equation}\label{e3 for filtration 1}
     e_3(\mathcal{I})\leq (e_2(\mathcal{I})-1)e_2(\mathcal{I}).
    \end{equation}    
    Furthermore, suppose $\mathcal{I}=\{I^n\}_{n\geq 0}$ and $I$ is an integrally closed ideal then,
    \begin{equation}\label{e3 integrally closed 1}
    e_3(I)\leq (e_2(I)-e_1(I)+e_0(I)-\ell(A/I))e_2(I).
    \end{equation}
    
    Suppose $d=3$ and equality holds in \eqref{e3 for filtration 1}  or \eqref{e3 integrally closed 1} for respective $\I$ then  $\depth G_{\wt{\mathcal{I}}}(A)\geq 2.$
 Suppose $d\geq 4$ and in addition to the above equalities for the respective $\mathcal{I}$, assume $e_k(\mathcal{I})=0$ for $4\leq k \leq d-1$. Then,
  \renewcommand{\labelenumi}{(\roman{enumi})}
    \begin{enumerate}
        \item $(-1)^de_d(\mathcal{I})\geq 0$ and 
        \item $e_d(\mathcal{I})=0$ implies $\depth G_{\wt{\mathcal{I}}}(A)\geq d-1.$
    \end{enumerate}
    \end{theorem}

In \cite[Theorem 3.12]{cpr}, Corso, Polini and Rossi   proved that if $I$ is normal and $e_2(I)=e_1(I)-e_0(I)+\ell(A/I)$, then $e_3(I)=0$. We recover this result as for normal ideals $e_3(I)\geq 0.$ Theorem \ref{theorem1} generalizes the fact
$e_2(I)\leq 1 \implies e_3(I)\leq 0$ proved in \cite[Proposition 6.4]{Put2} and \cite[Proposition 2.6]{mf2}.

It is natural to ask if Rossi's bound in \eqref{rossi bound} holds true for an $I$-admissible filtration. It is largely unknown even in dimension two. For the reduction number $\wt{r}_J(I)$  of Ratliff-Rush filtration $\{\wt{I^n}\}_{n\geq 0}$, Rossi and Swanson \cite[Proposition 4.5]{rs} proved that $\wt{r}_J(I)\leq r_J(I)$ in a two-dimensional Cohen-Macaulay local ring   and asked whether it is true for $d\geq 3$, \cite[Question 4.6]{rs}? Mandal and the first author in \cite[Proposition 2.1]{ms} proved that $\wt{r}_J(\mathcal{I})\leq r_J(\mathcal{I})$ for any $I$-admissible filtration $\mathcal{I}$ when $d=2$. In  Theorem \ref{theorem4}, we extend it in higher dimension.   In the same paper \cite[Section 4]{rs}, Rossi and Swanson asked whether $\rho(I)\leq r_J(I)$  where   $\rho(I)$  is the stability index defined below: 
$$\rho (I)= \min \{ n\in \mathbb{N}:\wt{I_k}=I_k \;\;\text{for all}\;\; k\geq n\}.
$$
Rossi, Trung and Trung \cite[Example 4.7]{rtt} gave a negative answer to this question. In this paper, we prove the following result.
 \begin{theorem}\label{theorem4}
 Let $(A, \mathfrak{m})$ be a Cohen-Macaulay local ring of dimension $d \geq 3$, $I$ an $\mathfrak{m}$-primary ideal and $\mathcal{I}$ an $I$-admissible filtration. If either of the following two conditions holds:
  \renewcommand{\labelenumi}{(\roman{enumi})}
\begin{enumerate}
    \item $e_3(\mathcal{I}) = e_2(\mathcal{I})(e_2(\mathcal{I}) - 1)$,
    \item $\mathcal{I}=\{I^n\}_{n\geq 0}$, $I$ is integrally closed and $e_3(I) = e_2(I)(e_2(I) - e_1(I) + e_0(I) + \ell(A/I))$,
\end{enumerate}
and $e_k(\mathcal{I}) = 0$ for $4 \leq k \leq d$.
Then the following holds:
 \renewcommand{\labelenumi}{(\roman{enumi})}
 \begin{enumerate}
        \item$\wt{r}_J(\mathcal{I})\leq r_J(\mathcal{I}).$
        \item  If $\wt{I_{r}}=I_{r}$ for some $r \geq r_J(\mathcal{I})$, then $\rho(\mathcal{I})\leq r.$ Furthermore, 
 \begin{eqnarray*}
            \rho(\mathcal{I})\leq \begin{cases}
                r_J(\mathcal{I}) &\text{ if } \wt{I_{r_J(\mathcal{I})}}=I_{r_J(I)}\\
                r_J(\mathcal{I})+(-1)^{d+1}(e_{d+1}(\mathcal{I})-e_{d+1}(\wt{\mathcal{I}})) &\text{ otherwise. } 
            \end{cases}
        \end{eqnarray*}
        
    \end{enumerate}
\end{theorem}
 As mentioned earlier, the vanishing of the error terms and the higher Hilbert coefficients has interesting consequences. The next theorem provides a partial answer of  \cite[Question 4.4]{nq} related to the
 Castelnuovo-Mumford regularity of $G(I)$.
\begin{theorem}\label{theorem5}
Let the hypothesis be the same as in Theorem \ref{theorem3} and 
$\wt{I^{r_J(I)}}=I^{r_J(I)}$. Then $\reg G_I(A)=r_J(I).$
\end{theorem}
 


This paper is organised in five sections beginning with a brief mention of notations in Section \ref{section-2}. In Section \ref{section-3}, we discuss the consequences of good behaviour of Ratliff-Rush filtration along with  a necessary and sufficient condition for the same. We also prove Theorem   \ref{red-bd} in this section. 
Theorem \ref{theorem1} is proved in Section \ref{section-4}. The last section of the paper consists of the proofs for Theorem \ref{theorem3}, Theorem \ref{theorem4} and Theorem \ref{theorem5}. Our results are supported by examples computed using Macaulay2/CoCoA. 

\section{notation}\label{section-2}
{\bf{I.}} 
Suppose $x_1,\ldots,x_s$ is a sequence in $A$. 
We write
 $(\underline{x})$ for the ideal $(x_1,\ldots,x_s)\subseteq A$. We use the notation $I^\prime$ and $A^\prime$ for going modulo a superficial element. For a graded module $M=\mathop\bigoplus\limits_{n\geq 0}M_n$, $M_n$ denotes its $n$-th graded piece and $M(-j)$ denotes the graded module $M$ shifted to the left by degree $j.$ 

{\bf{II.}} We write $\mathcal{F}=\{\wt{I_n}\}_{n\geq 0}$ and $\mathcal{F}^\prime=\Big\{\frac{\wt{I_n}+(x_1)}{(x_1)}\Big\}_{n\geq 0}$
where $x_1\in I$ is a superficial element. The Rees ring of  $\mathcal{I}$ is  $\R(\mathcal{I})=\mathop\bigoplus\limits_{n\geq 0} I_nt^n \subseteq A[t]$ where $t$ is an indeterminate and the associated graded ring of $\mathcal{I}$ is $G_{\mathcal{I}}(A)=\mathop\bigoplus\limits_{n\geq0}I_n/I_{n+1}$. We write  $\R(I)$ and $G_I(A)$ for the Rees ring and the associated graded ring respectively when $\mathcal{I}=\{I^n\}_{n\geq 0}$. Set ${\R(\I)}_+=\mathop\bigoplus\limits_{n\geq 1}I_nt^n$ and $\mathcal{M}=\m\mathop\bigoplus {\R(\I)}_+$. We write ${G}_{{\mathcal{F}}}(A)$ for the graded ring $\displaystyle{\bigoplus_{n\geq 0}}\frac{\wt{I_n}}{\wt{I_{n+1}}}$ for filtration $\mathcal{F}=\{\wt{I_n}\}_{n\geq 0}.$ 
 

{\bf{III.}} It is well known that the Hilbert series of $\mathcal I$, defined as $H(\mathcal I,t)=\mathop\sum\limits_{n\geq 0} \ell(I_n/I_{n+1})t^n$, is a rational function, i.e, there exists a unique rational polynomial $h_{\mathcal I}(t)\in\mathbb{Q}[t]$ with $h_{\mathcal I}(1)\neq 0$ such that $$H(\mathcal I,t)=\frac{h_{\mathcal I}(t)}{(1-t)^d}.$$ For every $i\geq 0$, we have $e_i(\mathcal I)=\frac{h_{\mathcal I}^{(i)}(1)}{i!}$, where $h_{\mathcal I}^{(i)}(1)$ denotes the $i$-th formal derivative of the polynomial $h_{\mathcal I}(t)$ at $t=1$. The integers $e_i(\mathcal I)$ are called the Hilbert coefficients of  $\mathcal {I}$ and for $0\leq i\leq d$, these are same as defined earlier in the Introduction. We refer to \cite{gr} for details.
We write ${e_i}(\wt{\mathcal{I}})$ for the coefficients $e_i(\mathcal{I})$ when $\mathcal{I}=\{\widetilde{I^n}\}_{n\geq 0}.$

{\bf{IV.}} The second Hilbert function of $\I$, defined as $H^2_{\I}(n)=\displaystyle{\sum^n_{j=0}}\ell(A/I_{j+1})$ coincides with a polynomial $P_\I^2(n)$ known as the second Hilbert polynomial for all sufficiently large values of $n.$ This is a polynomial of degree $d+1$, and can be formulated as 
$$
P_\I^2(n)=\sum_{j=0}^{d+1}(-i)^je_j(\I)\binom{n+d-j+1}{d-j+1}.
$$

{\bf{V.}} Let $M$ be a graded ${\R}(I)$-module and $H^i_\mathcal{M}(M)$ denote the $i$-th local cohomology module of $M$ with respect to $\mathcal{M}$. We define $a_i(M):=\max\{n~|~H^i_{\mathcal{M}}(M)_n\neq 0\}$. The Castelnuvo-Mumford regularity of $M$ is define as 
$
\reg (M)= \max\{a_i(M)+i:0\leq i \leq \dim M\}.
$

\section{Good behaviour of the Ratliff-Rush filtration}\label{section-3}
Suppose $(A,\m)$ is a Cohen-Macaulay local ring of dimension $d\geq 3.$ Marley \cite{tm}, proved that if $\depth G_{\I}(A)\geq d-1$, then $e_i(\I)\geq 0$ for $0\leq i \leq d$ and $e_i(\I)=0$ implies $e_j(\I)=0$ for $i\leq j\leq d.$ As evident by Example \ref{ex1.2}, both the above results may hold even if $\depth G_\I(A)=0.$ In Theorem \ref{conseq of behav of RREF}, we prove that in Marley's results, the depth condition on $G_\I(A)$ can be replaced by a weaker condition on the behaviour of the Ratliff-Rush filtration of $\I.$ This condition has interesting consequences which is the main content of this section. 

We first recall that the Ratliff-Rush filtration $\widetilde{\I}$  of an $I$-admissible filtration $\I$  is defined by letting $$\widetilde{I_n}=\mathop\bigcup\limits_{k\geq 0} (I_{n+k}:I^k).$$
Suppose $x\in I$ is a superficial element for $\I$, then it is easy to see that $\widetilde{I_n}A^\prime\subseteq \widetilde{I_nA^\prime}$ for all $n\geq 0$ and equality holds for $n\gg 0$ where $A^\prime=A/(x).$ Moreover, the equality for all $n\geq 0,$ is a very interesting condition. We will discuss that it is weaker than $\depth G_\I(A)\geq 1$, yet it enforces many nice properties.  This notion, for $I$-adic filtration, was first considered in the papers of Puthenpurakal \cite{Put1}, \cite{Put2} and was further considered in  \cite{mp}, \cite{Put3}, \cite{sy}. We generalize the definition from \cite{Put1} for an $I$- admissible filtration. 

\begin{definition}\label{def1}  
Let $x_1,\ldots,x_s\in I$ be a superficial sequence for $\I$. Then, we say that 
 \renewcommand{\labelenumi}{(\roman{enumi})}
\begin{enumerate}
    \item the Ratliff-Rush filtration of $\I$ behaves well modulo $x_1$ if    $\widetilde{I_n}A^\prime= \widetilde{I_nA^\prime}$
     for all $n\geq 0$ and 
    \item the Ratliff-Rush filtration of $\I$ behaves well modulo $x_1,\ldots,x_s$ if the Ratliff-Rush filtration of $I/(x_1,\ldots,x_{i-1})$ behaves well modulo the image of $x_i$ in $A/(x_1,\ldots,x_{i-1})$ for all $1 \leq i \leq s$, where $(x_1,\ldots,x_{i-1})=(0)$ for $i=1$.
\end{enumerate}
 \end{definition}
 \begin{remark}\label{rmk1}
  \renewcommand{\labelenumi}{(\roman{enumi})}
\begin{enumerate}
    \item Suppose $\depth G_\I(A)\geq d-1$. Then $H_\mathcal{M}^i(G_\I(A))=0$ for $0\leq i\leq d-2$. Using \cite[Proposition 5.2]{Put1} and \cite[Theorem 4.5]{Put2}, we get that the Ratliff-Rush filtration of $\I$ behaves well modulo a superficial sequence of length $d-2.$ In \cite[Theorem 6.2]{Put2}, Puthenpurkal proved that in a Cohen-Macaulay local ring if $e_2(I)=\ldots=e_d(I)=0$, then the Ratliff-Rush filtration of $I$ behaves well mod a superficial sequence of length $d-1$ which is the case for the following example. However $\depth G_I(A)=0.$
\item  When the Ratliff-Rush filtration of $\I$ behaves well modulo a superficial sequence $\underline{x}=x_1,\ldots$ $,x_{s}$, then ${G}_{{\mathcal{F}}}(A)/(\underline{x}){G}_{{\mathcal{F}}}(A)\iso {G}_{{\mathcal{F}}}(A_{s})$ where $A_s=A/(x_1,\ldots,x_s)$. Since  $\depth$ $ {G}_{{\mathcal{F}}}(A_{s})\geq 1$, by Sally's machine $\depth{G}_{{\mathcal{F}}}(A)\geq s+1.$ As an immediate consequence, we have Theorem \ref{conseq of behav of RREF} which implies the generalization of Marley's results \cite[Corollary 3.15 and Corollary 3.16]{tm}. 
\end{enumerate}
\end{remark}


\begin{example}\cite[Example 3.8]{cpr}\label{ex1.2}
Let $A=Q[[x,y,z]]$ and $I=(x^2-y^2,y^2-z^2,xy,yz,xz).$ Then, the Hilbert Series of $G_I(A)$ is $$
H(I,t)=\frac{5+6t^2-4t^3+t^4}{(1-t)^3}.
$$
So $e_1(I)= 4$ and $ e_2(I)=e_3(I)=0$. By \cite[Theorem 6.2]{Put2}, the Ratliff-Rush filtration of  $I$ behaves well modulo a superficial element. However, $\depth G_I(A)=0$ as $x^2\in(I^2:I)\subseteq\widetilde{I}$ but $x^2\notin I$.  
We also note that $J=(x^2,y^2,z^2)$ is a minimal reduction of $I$ and $r_J(I)=2=e_1(I)-e_0(I)+\ell(A/I)+1.$
\end{example}

\begin{theorem} \label{conseq of behav of RREF}
Let $(A,\m)$ be a Cohen-Macaulay local ring of dimension $d\geq 3$ and $I$ an $\m-$ primary ideal. Suppose $\mathcal{I}=\{I_n\}_{n\geq 0}$ is an $I$-admissible filtration.
     If the Ratliff-Rush filtration of $\I$ behaves well modulo a superficial sequence $x_1,\ldots,x_{d-2}$, then the following holds.
    
     \renewcommand{\labelenumi}{(\roman{enumi})}
\begin{enumerate}
        \item $e_i(\mathcal{I})\geq 0$ for all $0\leq i \leq d$.
        \item $e_i(\mathcal{I})=0$ $\implies e_j(\mathcal{I})=0$ for $i \leq j \leq d.$
         \item If $I_1$ is an integrally closed ideal, then $(-1)^d(e_0(\mathcal{I})-e_1(\mathcal{I})+\ldots+(-1)^d e_d(\mathcal{I})-\ell(A/I_1))\geq 0$.
    \end{enumerate}
\end{theorem}
\begin{proof}(i) and (ii):  It is well known that $e_0(\mathcal{I})\geq 0.$ Suppose the Ratliff-Rush filtration  of $\I$ behaves well modulo a superficial sequence $x_1,\ldots$ $,x_{d-2}$, then  $\depth {G}_{\mathcal{F}}(A)\geq d-1$ which implies that for $1\leq i\leq d,$
        \begin{equation}\label{e}      e_i(\I)={e}_i(\wt{\mathcal{I}})=\sum_{n\geq i-1}\binom{n}{i-1}\ell\Big(\frac{\wt{I_{n+1}}}{J \wt{I_n}}\Big)\geq 0.
        \end{equation}
        See \cite[Theorem 2.5]{rv} for the above formula. If $e_i(\mathcal{I})=0$ for some $1\leq i\leq d$, then by the above formula we get that $\ell\Big(\frac{\wt{I_{n+1}}}{J \wt{I_n}}\Big)=0$ for all   $n\geq i-1$ which means $\wt{r}_J(\mathcal{I})\leq i-1$. Hence $e_j(\mathcal{I})={e_j}(\wt{\I})=\mathop\sum\limits_{n\geq j-1}\binom{n}{j-1}\ell\Big(\frac{\wt{I_{n+1}}}{J \wt{I_n}}\Big)=0$ for all $j\geq i$.  This completes the proof of (i) and (ii).        
         
         
       (iii) We set $V_n=\ell\Big(\frac{\wt{I_{n+1}}}{J\wt{I_n}}\Big).$ We have 
 \begin{align*}
&(-1)^d\big(e_0(\mathcal{I})-e_1(\mathcal{I})+\ldots+(-1)^{d-1}e_{d-1}(\mathcal{I})-\ell (A/I_1)\big)\\
&=(-1)^d\big({e}_0(\wt{\mathcal{I}})-{e}_1(\wt{\mathcal{I}})+\ldots+(-1)^{d-1}{e}_{d-1}(\wt{\mathcal{I}})-\ell (A/I_1)\big)\\
&=(-1)^d \Big(\ell(A/J)-\sum_{n\geq 0}V_n+\sum_{n\geq 1} nV_n+\cdots+(-1)^{d-1}\sum_{n\geq d-2}\binom{n}{d-2}V_n-\ell(A/I_1)\Big)\\
&=(-1)^d \Big(\ell(I_1/J)+\sum_{n\geq d-2}[-1+\binom{n}{1}+\cdots+(-1)^{d-1}\binom{n}{d-2}]V_n-\sum_{n=0}^{d-3}V_n+\sum_{n=1}^{d-3}nV_n-\\
&\cdots+(-1)^{d-3}\sum_{n=d-4}^{d-3}\binom{n}{d-4}V_n+(-1)^{d-2}\binom{n}{d-3}V_{d-3}\Big)\\
&=(-1)^{d}\Big(\ell(I_1/J)+(-1)^{d-1}\sum_{n\geq d-2}\binom{n-1}{d-2}V_n-\sum_{n=0}^{d-3}\big( \sum_{k=0}^{n}(-1)^{k}\binom{n}{k} \big)V_n\Big)\\
&= (-1)^{d}\Big(\ell(I_1/J)+(-1)^{d-1}\sum_{n\geq d-2}\binom{n-1}{d-2}V_n-V_0\Big)\\
&= (-1)^{d}\Big(-\ell(\wt{I_1}/I_1)+(-1)^{d-1}\sum_{n\geq d-1}\binom{n-1}{d-2}V_n\Big)\\
&=(-1)^{d}\Big(-\ell(\wt{I_1}/I_1)-(-1)^d\sum_{n\geq d-1}\Big(\binom{n}{d-1}-\binom{n-1}{d-1}\Big)V_n\Big)\\
&=(-1)^{d}\Big(-\ell(\wt{I_1}/I_1)-(-1)^de_d(\I)+(-1)^d\sum_{n\geq d-1}\binom{n-1}{d-1}V_n\Big).
     \end{align*}
Hence,
\begin{align}
 \begin{split}
    &(-1)^d\big(e_0(\mathcal{I})-e_1(\mathcal{I})+\ldots+(-1)^{d-1}e_{d-1}(\mathcal{I})+(-1)^de_d(\I)-\ell (A/I_1)\big)\nonumber\\
   &=(-1)^{d}\Big(-\ell(\wt{I_1}/I_1)+(-1)^d\sum_{n\geq d-1}\binom{n-1}{d-1}V_n\Big).
    \end{split}
\end{align}
Since $I_1$ is an integrally closed ideal, hence
\begin{align*}
    &(-1)^d\big(e_0(\mathcal{I})-e_1(\mathcal{I})+\ldots+(-1)^{d-1}e_{d-1}(\mathcal{I})+(-1)^de_d(\I)-\ell (A/I_1)\big)\\
    &=\sum_{n\geq d-1}\binom{n-1}{d-1}V_n\geq 0.
\end{align*}
 \end{proof}


We recover Marley's results \cite[Corollary 3.15 and Corollary 3.16]{tm}. 

\begin{corollary}
    Suppose $\depth G_\I(A)\geq d-1$. Then the following holds.
     \renewcommand{\labelenumi}{(\roman{enumi})}
     \begin{enumerate}
        \item $e_i(\mathcal{I})\geq 0$ for all $0\leq i \leq d$.
        \item $e_i(\mathcal{I})=0$ $\implies e_j(\mathcal{I})=0$ for $i \leq j \leq d.$
         \item If $I_1$ is an integrally closed ideal, then$(-1)^d(e_0(\mathcal{I})-e_1(\mathcal{I})+\ldots+(-1)^d e_d(\mathcal{I})-\ell(A/I_1))\geq 0.$ 
         
    \end{enumerate}
\end{corollary}
\begin{proof}
See Remark \ref{rmk1} and Theorem \ref{conseq of behav of RREF}.
\end{proof}
Another noteworthy consequence of the good behaviour of Ratliff-Rush is the following result. 
\begin{theorem}\label{rossi's holds}
Suppose $d\geq 3.$ If the Ratliff-Rush filtration of $I$ behaves well modulo a superficial sequence $x_1,\ldots,x_{d-2}$, then the Rossi's bound for reduction number holds i.e, $r_J(I)\leq e_1(I)-e_0(I)+\ell(A/I)+1$.
\end{theorem}
\begin{proof}
Under the hypothesis, $\depth {G}_{\mathcal{F}}(A)\geq d-1$,  so using \cite[Theorem 1.3]{r}, we have $r_J(I)\leq \displaystyle{\sum_{n\geq 0}}\ell(\wt{I^{n+1}}/J\wt{I^n})-e_0(I)+\ell(A/I)+1=e_1(I)-e_0(I)+\ell(A/I)+1$. 
\end{proof}

As mentioned already, Puthepurkal \cite{Put2} proved that  $e_2(I)=\ldots=e_d(I)=0$, implies the good behaviour of the Ratliff-Rush filtration of $I$. However, for integrally closed ideals, the above vanishings are quite strong as by a result of Itoh  \cite{itoh}, $e_2(I)=0\implies $ $G_I(A)$ is Cohen-Macaulay.  We prove the following result for integrally closed ideals which provides a sufficient condition for good behaviour of Ratliff-Rush filtration. Example \ref{eg1} demonstrates that $e_2(I)=e_1(I)-e_0(I)+\ell(A/I)$ and $e_3(I)=0$ is a weaker condition than  $G_I(A)$ being Cohen-Macaulay. Let us first define 
 $$L^I(A)=\mathop\bigoplus\limits_{n\geq 0} A/I^{n+1}$$ as in \cite{Put1} and $\mathcal{R}(I)=\displaystyle{\bigoplus_{n\geq 0}}I^nt^n$, the Rees algebra of $I$. The exact sequence 
$$0\to \R(I)\to A[t]\to L^I(A)(-1)\to 0$$ defines an $\R(I)$-module structure on $L^I(A).$ Note that $L^I(A)$ is not a finitely generated $\R(I)$-module. However, the local cohomology modules $H_{\mathcal{M}}^i (L^I(A))$, with support in $\mathcal{M}=m\bigoplus \R(I)_+$,
are $*$-Artinian for $0\leq i\leq 2$ and $H^0_\mathcal{M}(L^I(A))=\displaystyle{\bigoplus_{n\geq 0}}\frac{\wt{I^{n+1}}}{I^{n+1}}$ see \cite[section 4]{Put1}. For a superficial element $x\in I$, let $B(x,I)=\displaystyle{\bigoplus_{n\geq 0}}\frac{I^{n+1}:x}{I^n}$ and consider the exact sequence of $\R(I)$ modules \begin{align}\label{LI-seq} 0 \to B(x,I) \to L^I(A)(-1)\xrightarrow{x}L^I(A)\xrightarrow{\rho}L^{I^\prime}(A^\prime)\to 0 \end{align}
which is induced by the exact sequence of $A$-module $$0\to \frac{I^{n+1}:x}{I^n}\to A/I^n\xrightarrow{x} A/I^{n+1}\xrightarrow{\rho} A^\prime/I{^\prime}^{n+1}\to 0.$$
From the sequence \eqref{LI-seq}, we get the exact sequence of local cohomology modules 
  \begin{equation}\label{first_funda_ sequence}
      \begin{aligned}
      0\to B(x,I)\to & H_{\mathcal{M}}^0(L^I(A))(-1){\longrightarrow}H_{\mathcal{M}}^0(L^I(A)){\longrightarrow}H_{\mathcal{M}}^0(L^{I^\prime}(A^\prime))\\
       &H_{\mathcal{M}}^1(L^I(A))(-1){\longrightarrow}H_{\mathcal{M}}^1(L^I(A)){\longrightarrow}H_{\mathcal{M}}^1(L^{I^\prime}(A^\prime))\\
         &\cdots\\
  &  {\longrightarrow}H_{\mathcal{M}}^{d-1}(L^I(A))(-1){\longrightarrow}H_{\mathcal{M}}^{d-1}(L^I(A)){\longrightarrow}H_{\mathcal{M}}^{d-1}(L^{I^\prime}(A^\prime))\\
  &{\longrightarrow}H_{\mathcal{M}}^{d}(L^I(A))(-1){\longrightarrow}H_{\mathcal{M}}^{d}(L^I(A)){\longrightarrow}H_{\mathcal{M}}^{d}(L^{I^\prime}(A^\prime)).    
      \end{aligned}
  \end{equation}
We recall the following result from \cite{Put2} for a quick reference. 
\begin{theorem}\cite[Theorem 4.5 ]{Put2}\label{4.5pu2}
Let $(A,\m)$ be a Cohen-Macaulay local ring of dimension $d\geq 2.$ Then the Ratliff-Rush filtration of $I$ behaves well modulo superficial sequence $x_1,\ldots,x_{s}$ if and only if $\ell(H_\mathcal{M}^i(L^I(A))=0$ for $i=1, \ldots, s. $
\end{theorem}

In \cite[Theorem 3.12]{cpr}, authors proved that if $I$ is normal and $e_2(I)=e_1(I)-e_0(I)+\ell(A/I)$, then $G_I(A)$ is Cohen-Macaulay and $e_k(I)=0$ for $k\geq 3$. The following result can be seen as a slight generalization of the above for integrally closed ideals. 
\begin{theorem}\label{rrf}
     Let $(A,\m)$ be a Cohen-Macaulay local ring of dimension $d\geq 2$ and $I$ an $\m$- primary integrally closed ideal. Suppose  $e_2(I)=e_1(I)-e_0(I)+\ell(A/I)$ and $e_k(I)=0$ for all $3 \leq k \leq d$. Then, the following holds \renewcommand{\labelenumi}{(\roman{enumi})}.
     
     \begin{enumerate}
         \item $H_\mathcal{M}^i(L^I(A))=0$ for $i= 1,\ldots,d-1$. \label{H^i}
         \item \label{part2}  
         The Ratliff-Rush filtration of $I$ behaves well modulo a superficial sequence of length $d-1.$ 
         \item $G_I(A)$ is generalized Cohen-Macaulay, $\depth G_I(A)\in \{0,d\}$ and \label{buchs}
         $$
         \mathbb{I} (G_I(A)) = (-1)^{d+1}d\cdot e_{d+1}(I)
         $$
         where $\mathbb{I}(G_I(A))=\displaystyle{\sum_{i=0}^{d-1}}\binom{d-1}{i}\ell(H_\mathcal{M}^i(G_I(A)).$
         \item \label{part3} If $I=\m$, then $G_\m(A)$ is a quasi-Buchsbaum ring.
         \end{enumerate}
\end{theorem}
\begin{proof}
 \renewcommand{\labelenumi}{(\roman{enumi})}
\begin{enumerate}    
  \item  We apply induction on $d$. For $d=2$, by  \cite[Proposition 3.4]{Put2}, $\widetilde{I^n}A^\prime=\widetilde{I^nA^\prime}$ for all $n\geq 1$ which implies that $H_\mathcal{M}^1(L^I(A))=0$ using Theorem \ref{4.5pu2}. 

Now we assume $d\geq 3$ and that the result holds for $d-1.$ Let $x\in I$ be a superficial element such that $I/(x)$ is an integrally closed ideal. Set  $A^\prime =A/(x)$, then $\dim A^\prime=d-1$, $e_i(I)=e_i(I^\prime)$ for $0\leq i \leq d-1$ and $\ell(A/I)=\ell (A^\prime/I^\prime)$.
    It is given that $e_2(I)=e_1(I)-e_0(I)+\ell (A/I)$ and $e_k(I)=0$ for $3\leq k \leq d$. Hence, $e_2(I^\prime)=e_1(I^\prime)-e_0(I^\prime)+\ell (A^\prime/I^\prime)$ and  $e_k(I^\prime)=0$ for $3\leq k \leq d-1$. Therefore by induction hypothesis $H^i_\mathcal{M}(L^{I^\prime}(A^\prime))=0$ for $i= 1, \ldots d-2.$ So, from the exact sequence \eqref{first_funda_ sequence}, $H_\mathcal{M}^i(L^I(A)=0$ for $i= 2,\ldots, d-1.$ 
    Also, the Ratliff-Rush filtration of $I^\prime$ behaves well modulo a superficial sequence of length $d-2$, see Theorem \ref{4.5pu2}. So, as mentioned already, $\depth {G}_{\wt{I^\prime}}(A^\prime)\geq d-1$. Now we show that 
    $H_\mathcal{M}^1(L^I(A)=0.$
 Using the formula given in \eqref{e}, we get
\begin{align*}
&e_2({I}^\prime)-e_1({I}^\prime)+e_0({I}^\prime)-\ell({A}^\prime/{I}^\prime)\nonumber\\
     &={e_2}(\wt{I^\prime})-{e_1}(\wt{I^\prime})+{e_0}(\wt{I^\prime})-\ell({A}^\prime/{I}^\prime)\label{the-equality-for-f-1}\\
   &=\sum_{n\geq 1}n\ell (\widetilde{{{I}^\prime}^{n+1}}/J^\prime\widetilde{{{I}^\prime}^n})-\sum_{n\geq 0}\ell(\widetilde{{{I}^\prime}^{n+1}}/J^\prime\widetilde{{{I}^\prime}^{n}})+\ell(R^\prime/J^\prime)-\ell({R}^\prime/{I}^\prime)\\ 
 &= \sum_{n\geq 2}(n-1)\ell (\widetilde{{{I}^\prime}^{n+1}}/J^\prime\widetilde{{{I}^\prime}^n})-(\ell({{I}^\prime}/J^\prime)+\ell({R}^\prime/{I}^\prime))+\ell({R}^\prime/J^\prime)\nonumber\\
&=\sum_{n\geq 2}(n-1)\ell (\widetilde{{{I}^\prime}^{n+1}}/J^\prime\widetilde{{I^\prime}^n})=0.
\end{align*}
Therefore $\widetilde{{I^\prime}^{n+1}}=J^\prime \widetilde{{I^\prime}^{n}}$ for all $n\geq 2$ which gives ${e_d}(\wt{I^\prime})=0.$ Therefore, we have 
\begin{equation}\label{*}
 e_{d}(I^\prime) = {e_d}(\wt{I^\prime})+ (-1)^d \sum_{j \geq 1} \ell
 \left(\frac{\widetilde{{I^\prime}^j }}{{I^\prime}^j }\right) = (-1)^d \lambda (H_{\mathcal{M}}^0
 (L^{I^\prime} (A^\prime))). 
 \end{equation}
See \cite[1.5]{Put2} for the first equality.
On the other hand, by Singh's formula \cite[Proposition 1.2(4)]{rv},  \begin{equation*}
 e_{d}(I^\prime)=e_d(I)+(-1)^d \lambda (B(x,I)).
 \end{equation*}
Since $e_d(I)=0$, we get $\lambda (H_{\mathcal{M}}^0
 (L^{I^\prime} (A^\prime)))=\lambda (B(x,I)).$
Now from the exact sequence \eqref{first_funda_ sequence}, 
$ H_{\mathcal{M}}^1 (L^I (A))(-1) \rightarrow H_{\mathcal{M}}^1 (L^I (A))$ is an injective map of $*$-Artinian modules which implies that $H_{\mathcal{M}}^1 (L^I (A))=0$. This completes the proof. 

 \item It follows from part (\ref{H^i}) and Theorem \ref{4.5pu2}.
 
 \item By Remark \ref{rmk1}, $\depth {G}_{\mathcal{F}}(A)\geq d.$ Now consider the first fundamental exact sequence as defined in \cite{Put1}, 
 $$0\to G_I(A) \to L^I(A) \to L^I(A)(-1)\to 0$$
 and the induced exact sequence of local cohomology modules
 \begin{align}\label{first_fundamental}
   0\to H_{\mathcal{M}}^0(G_I(A))\to H_{\mathcal{M}}^0(L^I(A))\to H_{\mathcal{M}}^0(L^I(A))(-1)\to H_{\mathcal{M}}^1(G_I(A))\to\cdots.
 \end{align}
This gives $H_{\mathcal{M}}^i(G_I(A))=0$ for $2\leq i\leq d-1$. It follows from  Claim \ref{claim1} that 
$$
H_\mathcal{M}^0(G_I(A))=\bigoplus_{i\geq 0} \frac{\wt{I^{i+1}}\bigcap I^i}{I^{i+1}}=\bigoplus_{i\geq 0}\frac{\wt{I^{i+1}}}{I^{i+1}}=H_\mathcal{M}^0(L^I(A)).
$$  
Consequently $H_{\mathcal{M}}^1(G_I(A)) \cong H_\mathcal{M}^0(L^I(A))(-1)=\displaystyle{\bigoplus_{i\geq 0}}\frac{\wt{I^{i}}}{I^{i}}.$  Hence $G_I(A)$ is generalised Cohen-Macaulay, $\depth G_I(A)=0$ or $d$ and $$\mathbb{I}(G_I(A)) = \ell(H_{\mathcal{M}}^0(G_I(A)))+(d-1) \ell(H_{\mathcal{M}}^1(G_I(A))) =d \cdot\ell(H_{\mathcal{M}}^0(L^I(A))).$$

We have $e_2(I)-e_1(I)+e_0(I)-\ell(A/I)=\displaystyle{\sum_{n\geq 2}}(n-1)\ell (\widetilde{{{I}}^{n+1}}/J\widetilde{{{I}}^n})=0 $ which implies $\widetilde{I^{n+1}}=J\widetilde{I^n}$ for all $n\geq 2$. Since ${G}_{\mathcal{F}}(A)$ is Cohen Macaulay which implies  ${e}_{d+1}(\wt{I})=\mathop\sum\limits_{n\geq d}\binom{n}{d}\ell(\wt{I^{n+1}}/J\wt{I^n}) =0.$ Further, $e_{d+1}(I)={e}_{d+1}(\wt{I})+(-1)^{d+1}\ell(H_\mathcal{M}^0(L^I(A))),$ see \cite[1.5]{Put2}.  
 It follows that  $$\mathbb{I}(G_I(A))=d\cdot\ell(H_\mathcal{M}^0(L^I(A)))=   (-1)^{d+1}d\cdot e_{d+1}(I). $$

\textbf{Claim:}\label{claim1} $\widetilde{I^{n+1}}\subseteq I^n$ for $n\geq 1$.
\begin{proof}[Proof of the Claim] 
 We prove the claim by induction. For $n=1,$
 $\widetilde{I^2} \subseteq \widetilde{I}=I$ since $I$ is integrally closed. For $n\geq 2,$ we have
$\widetilde{I^{n+1}} = J\widetilde{I^{n}}$ and by induction hypothesis $J\widetilde{I^{n}}\subseteq J I^{n-1}\subseteq I^{n}.$
\end{proof}
\item We may assume that $G_{\m}(A)$ is not Cohen-Macaulay. In that case, $\mathcal{M}\cdot H_{\mathcal{M}}^i(G_I(A))=0$ for $0\leq i\leq d-1.$ 
\end{enumerate}
\end{proof}

We end this section by citing an example in support of the results of this section. 
\begin{example}\cite[Theorem 5.2]{or}\label{eg1}
Let $m\geq 0$, $d\geq 2$ and $k$ be an infinite field. Consider the power series ring $D=k[[\{X_j\}_{1\leq j\leq m}, Y, \{V_j\}_{1\leq j\leq d}, \{Z_j\}_{1\leq j\leq d}]]$ with $m+2d+1$ indeterminates and the ideal $\mathfrak{a}=[(X_j~|~1\leq j\leq m)+(Y)].[(X_j~|~1\leq j\leq m)+(Y)+(V_i~|~1\leq i\leq d)]+(V_iV_j~|~1\leq i,j\leq d, i\neq j)+(V_i^3-Z_iY~|~1\leq i\leq d).$ Define $A=D/\mathfrak{a}$ and $x_i,y,v_i,z_i$ denote the images of $X_i, Y, V_i, Z_i$ in $A$ respectively. Let $\mathfrak{m}=(x_j~|~1\leq j\leq m)+(y)+(v_j~|~1\leq j\leq d)+(z_j~|~1\leq j\leq d)$ be the maximal ideal in $A$ and $Q=(z_j~|~1\leq j\leq d).$ Then
 \renewcommand{\labelenumi}{(\roman{enumi})}
\begin{enumerate}
    \item $A$ is Cohen-Macaulay local ring with $\dim A=d,$
    \item $Q$ is a minimal reduction of $\m$ with $r_Q(\m)=3,$ 
    \item $G_\m(A)$ is Buchsbaum ring with $\depth G_\m(A)=0$ and  $\mathbb{I} (G_\m(A)) =d, $
    \item The Hilbert series $H(\m,t)$ of $A$ is given by
    $$
    H(\m,t)=\frac{1+\{m+d+1\}t+\displaystyle{\sum_{j=3}^{d+2}(-1)^{j-1}}\binom{d+1}{j-1}t^j}{(1-t)^d}
    $$
    \item $e_0(\m)=m+2d+2;$ $e_1(\m)=m+3d+2$; $e_2(\m)=d+1$ and $e_i(\m)=0$ for $3\leq i\leq d.$   
\item  \label{rref}Here $e_2(\m)-e_1(\m)+e_0(\m)-\ell (A/\m)=0$. Thus by Theorem \ref{rrf}, the Ratliff-Rush filtration of $\m$ behaves well modulo a superficial sequence of length $d-1$. In \cite[Example 5.1]{sy}, authors indicated that the Ratliff-Rush filtration of $I$ behaves well mod one superficial element which is now improved.
\item  Since $e_3(\m)=e_2(\m)(e_2(\m)-e_1(\m)+e_0(\m)-\ell(A/\m))$ and $\wt{\m^3}=\m^3$ by \cite{or}. Hence by Theorem \ref{cleto3}, we get $\reg G_\m(A)=r_Q(\m)=3.$

  \end{enumerate} 
\end{example}


\section{Bounds on the Higher Hilbert coefficients}\label{section-4}
In this section, we prove an upper bounds on the third Hilbert coefficient. One of the bounds is for an $I$-admissible filtration while the other is for $\m$-primary integrally closed ideals. The boundary conditions along with vanishing of the higher coefficients forces high depth of  $G_{{\mathcal{F}}}(A)$. 

\begin{theorem}\label{e3_m-primary}
 Let $(A,\m)$ be a Cohen-Macaulay local ring of dimension $d\geq 3$ and $I$ an $\m$- primary ideal. Suppose $\mathcal{I}=\{I_n\}_{n\geq 0}$ is an $I$-admissible filtration. Then 
    \begin{equation}\label{e3 for filtration}
     e_3(\mathcal{I})\leq (e_2(\mathcal{I})-1)e_2(\mathcal{I}).
    \end{equation}    
    Furthermore, suppose $\mathcal{I}=\{I^n\}_{n\geq 0}$ and $I$ is an integrally closed ideal then,
    \begin{equation}\label{e3 integrally closed}
    e_3(I)\leq (e_2(I)-e_1(I)+e_0(I)-\ell(A/I))e_2(I).
    \end{equation}
    
Suppose $d=3$ and equality holds in \eqref{e3 for filtration}  or \eqref{e3 integrally closed} for respective $\I$ then  $\depth G_{\mathcal{F}}(A)\geq 2.$

\end{theorem}
\begin{proof}
Suppose $d=3$. Let $\mathcal F=\{\widetilde{I_n}\}$ and $x\in I$ be a superficial element for $I$. Then $x$ is also a superficial element for the filtration $\mathcal F$. Suppose $J=(x,y,z)$ is a minimal reduction of $I$. Let $A^\prime=A/(x)$ and ${\mathcal F}^\prime=\{{\mathcal F}^\prime_n=\frac{\widetilde{I_n}+(x)}{(x)}\}$. Since $\depth G_\mathcal{F}(A)\geq 1$, hence by Singh's formula see \cite[Proposition 1.2]{rv}, we get $e_3(\mathcal{I})=e_3(\mathcal{F})=e_3(\mathcal{F}^\prime).$
Further, by \cite[1.5]{Put2}, we have
\begin{equation}\label{e3}
{e}_3(\wt{ {\mathcal F}^\prime})-e_3(\mathcal{F}^\prime)= \sum_{n\geq 0} \ell\left (\frac{\widetilde{{\mathcal F}^\prime_{n+1}}}{{\mathcal F}^\prime_{n+1}}\right )\geq 0.
\end{equation}
By \cite[Proposition 4.4]{bl},  we have
for all $n\geq -1$,
\begin{equation} \label{diff}
P_{\widetilde{{\mathcal F}^\prime}}(n)-H_{\widetilde{{{\mathcal F}^\prime}}} (n)=\ell((H^2_{\mathcal R_+}(\mathcal R(\widetilde{{ {\mathcal F}^\prime}})))_{n+1}). 
\end{equation}
Now, taking the sum for $m\gg 0$ on both sides of the above equation, we get
\begin{eqnarray*}
\sum_{n=0}^m \ell((H^2_{\mathcal R_+}(\mathcal R(\widetilde{{ {\mathcal F}^\prime}})))_{n+1})&=& \sum_{n=0}^m P_{\widetilde{{\mathcal F}^\prime}}(n)-\sum_{n=0}^m H_{\widetilde{{{\mathcal F}^\prime}}} (n)\\
&=& \sum_{n=0}^m P_{\widetilde{{\mathcal F}^\prime}}(n)- H^2_{\widetilde{{{\mathcal F}^\prime}}} (m)
\\
&=& {e}_0(\wt{{\mathcal F}^\prime}){m+3\choose 3}-{e}_1(\wt{{\mathcal F}^\prime}){m+2\choose 2}+{e}_2(\wt{{\mathcal F}^\prime}){m+1\choose 1}- P^2_{\widetilde{{{\mathcal F}^\prime}}} (m)\\
&=&{e}_3(\wt{{\mathcal F}^\prime}).
\end{eqnarray*}
Since $A^\prime$ is a $2$-dimensional Cohen-Macaulay local ring,  we have $$\ell((H^2_{\mathcal R_+}(\mathcal R(\widetilde{{ {\mathcal F}^\prime}})))_n)\leq \ell((H^2_{\mathcal R_+}(\mathcal R(\widetilde{{ {\mathcal F}^\prime}})))_{n-1})$$ for all $n\in \mathbb Z$ by \cite[Lemma 4.7]{bl}. Now in equation \eqref{diff}, we substitute $n=-1$ to get $$\ell((H^2_{\mathcal R_+}(\mathcal R(\widetilde{{ {\mathcal F}^\prime}})))_0)={e}_2({ \wt{{\mathcal F}^\prime}})=e_2({\mathcal F}^\prime)=e_2(\mathcal F)=e_2(\mathcal{I}).$$
Then,
\begin{equation}
{e}_3(\wt{ {\mathcal F}^\prime})=\sum_{n=0}^m \ell((H^2_{\mathcal R_+}(\mathcal R(\widetilde{{ {\mathcal F}^\prime}})))_{n+1})\leq\sum_{n=0}^ {a_2(\mathcal R(\widetilde{ {\mathcal F}^\prime}))-1}   \ell((H^2_{\mathcal R_+}(\mathcal R(\widetilde{{ {\mathcal F}^\prime}})))_0)= a_2(\mathcal R(\widetilde{ {\mathcal F}^\prime}))e_2(\mathcal{I})\nonumber.
\end{equation}
where $a_2(\mathcal R(\widetilde{ {\mathcal F}^\prime}))\leq a_2(G_{\widetilde{ {\mathcal F}^\prime}}(A^\prime)).$  By \cite[Corollary 5.7(2)]{tm} and \cite[Theorem 2.2]{ms} respectively, $a_2(G_{\widetilde{ {\mathcal F}^\prime}}(A^\prime))=\widetilde{r}({\mathcal F}^\prime)-2\leq e_2({\mathcal F}^\prime)+1-2=e_2(\mathcal{I})-1$.
Hence by using \eqref{e3}, we get
\begin{equation}\label{e3tilde}
    e_3(\mathcal{I})=e_3({\mathcal{F}}^\prime)\leq {e}_3(\wt{ {\mathcal F}^\prime}) \leq (\widetilde{r}({\mathcal F}^\prime)-2)e_2(\mathcal{I})\leq e_2(\mathcal{I})(e_2(\mathcal{I})-1).
\end{equation}
Furthermore, Suppose $I$ is an integrally closed ideal and $\mathcal{I}=\{I^n\}.$ Then by \cite[Lemma 3.8]{sy}, we have that $\widetilde{r}({\mathcal F}^\prime)\leq e_2({\mathcal F}^\prime)-e_1({\mathcal F}^\prime)+e_0({\mathcal F}^\prime)-\ell(A/I)+2.$ Therefore by similar argument using \eqref{e3tilde}, we get \begin{equation}\label{e3int}
    e_3(I)\leq e_3(\mathcal{F}^\prime)\leq (e_2(I)-e_1(I)+e_0(I)-\ell(A/I))e_2(I).
\end{equation}
Now, suppose equality holds either in \eqref{e3tilde} or \eqref{e3int}, then $e_3({\mathcal{F}}^\prime)= {e}_3( \wt{{\mathcal F}^\prime})$ which gives $\depth G_{\mathcal{F}^\prime}(A^\prime)\geq 1$ using \eqref{e3}. Therefore by Sally's machine, see \cite[Lemma 1.4]{rv}, we get $\depth G_\mathcal{F}(A)\geq 2$.

Now for $d\geq 4$, let $x\in I$ be a superficial element for $\mathcal{I}$ and $\mathcal{I}^\prime=\{\frac{I_n+(x)}{(x)}\}.$ Then $e_i(\mathcal{I}^\prime)=e_i(\mathcal{I})$ for $0\leq i \leq 3.$   
 Therefore, by induction,
 \begin{equation*}
     \begin{split}
         e_3(\mathcal{I}) =e_3(\mathcal{I}^\prime) &\leq e_2(\mathcal{I}^\prime)(e_2(\mathcal{I}^\prime)-1) \\ 
 &=e_2(\mathcal{I})(e_2(\mathcal{I})-1).
     \end{split}
 \end{equation*}
 
 Similarly, let $x\in I$ be a superficial element of $I$ such that $I^\prime= I/(x)$ is an integrally closed ideal in $A/(x)$. Then $e_i(I^\prime)=e_i(I)$ for $0\leq i \leq 3.$ Therefore, by induction
 \begin{equation*}
   \begin{split}
e_3(I)=e_3(I^\prime)&\leq (e_2(I^\prime)-e_1(I^\prime)+e_0(I^\prime)-\ell(A/I))e_2(I^\prime) \\
 &=(e_2(I)-e_1(I)+e_0(I)-\ell(A/I))e_2(I)
 \end{split}
 \end{equation*}
 \end{proof}
The immediate byproduct of Theorem \ref{e3_m-primary} are the following results from literature. 
\begin{corollary}
    Let $(A,\m)$ be a Cohen-Macaulay local ring of dimension $d\geq3.$ Let $I$ an $\m$-primary ideal and $J$ a minimal reduction of $I$. 
     \renewcommand{\labelenumi}{(\roman{enumi})}
    \begin{enumerate}
        \item (\cite[Proposition 6.4]{Put2} and \cite[Proposition 2.6]{mf2}) If $e_2(I)\leq 1$, then $e_3(I)\leq 0$.
        \item (\cite[Theorem 2.8]{mf2})  If $I$ an $\m$-primary integrally closed ideal and $e_2(I)=e_1(I)-e_0(I)+\ell(A/I)$, then $e_3(I)\leq 0.$
    \end{enumerate}  
\end{corollary}
\begin{example}\label{eg3}
    Let $A=k[[x,y,z]]$ and $I=(x^2-y^2,y^2-z^2,xy,yz)$. By CoCoA, the Hilbert series is
    $$
    H(I,t)=\frac{6+t+t^3}{(1-t)^3}.
    $$
    Then $e_2(I)=3$ and $e_3(I)=1\leq 6=e_2(I)(e_2(I)-1)$. 
\end{example}
The bound on Theorem \ref{e3_m-primary}, for integrally closed ideals, may not hold for non-integrally closed $\m$-primary ideals.
\begin{example}\cite[Example 3.1]{sy}
   Let $A= k[x,y,z]_{(x,y,z)}$ and $I=(x^4, y^4, z^4, x^3y, xy^3, y^3z,$ $ yz^3,x^2yz,xy^2z).$ Then $\depth G_I(A)=0$, the Hilbert series of $I$ is
    $$    H(I,t)=\frac{29+19t+19t^2-2t^3-2t^4+t^5}{(1-t)^3}
    $$
    and the Hilbert polynomial is
    $$
   P_I(n)=64{\binom{n+2}{3}}-48{\binom{n+1}{2}}+11{\binom{n}{1}}.
    $$
   Therefore, $e_2(I)-e_1(I)+e_0(I)-\ell (R/I)=-2$ and $e_3(I)=0.$ Hence, $e_3(I)> e_2(I)(e_2(I)-e_1(I)+e_0(I)-\ell(A/I))$, however  $e_3(I)\leq e_2(I)(e_2(I)-1).$  
\end{example}
Now, we attempt to generalize the last part of Theorem \ref{e3_m-primary} in higher dimension. For this, we need the following lemma. 

\begin{lemma}\label{e_d-1 in d dimen}
    Let $(A,\m)$ be a $d\geq 4$-dimensional Cohen-Macaulay local ring, $I$ an $\m$-primary ideal and $\mathcal{I}=\{I_n\}$ an $I$-admissible filtration. If $e_3(\mathcal{I})=e_2(\mathcal{I})(e_2(\mathcal{I})-1)$ and $e_k(\mathcal{I})=0$ for $4\leq k \leq d-1$ then $\depth {G}_{\wt{\mathcal{F}^\prime}}(A^\prime)\geq d-2$ where $\mathcal{F}= \{\wt{I_n}\}_{n\geq 0}$ and ${\mathcal{F}^\prime}=\{\frac{\wt{I_n}+(x_1)}{(x_1)}\}_{n\geq 0}$ for a superficial element $x_1$ of $\mathcal{I}.$
\end{lemma}
\begin{proof}
 We use induction on $d$. First, assume $d=4$.
    Suppose $J=(x_1,x_2,x_3,x_4)$ is a minimal reduction of $\mathcal{I}$. We write $A^\prime=A/(x_1)$ and $\mathcal{F}^\prime=\Big\{\frac{\wt{I_n}+(x_1)}{(x_1)}\Big \}_{n\geq 0}$. Then $e_i(\mathcal{F}^\prime)=e_i(\mathcal{I})$ for  $0\leq i\leq 3.$ Thus $e_3(\mathcal{F}^\prime)=e_2(\mathcal{F}^\prime)(e_2(\mathcal{F}^\prime)-1).$  By Theorem \ref{e3_m-primary}, we obtain that $\depth G_{\wt{\mathcal{F}^\prime}}(A^\prime)\geq 2.$
Now we  prove the result for $d\geq 5$ and assume that it holds true for dimension $d-1$. 
 Suppose $J=(x_1,\ldots,x_d)$ is a minimal reduction of $\mathcal{I}$. Set $\mathcal{F}^\prime=\Big\{\frac{\wt{I_n}+(x_1)}{(x_1)}\Big\}_{n\geq 0}$ which is an $I/(x_1)$-admissible filtration in $A^\prime=A/(x_1)$. For convenience, we write $\mathcal{J}$ for the filtration $\wt{(\mathcal{F}^\prime)}$. By induction hypothesis, 
  \begin{equation}\label{depth tilde}
\depth G_{ \wt{{\mathcal J}^\prime}}(A^{\prime\prime})\geq d-3.
 \end{equation}
 where $A^{\prime\prime}=A/(x_1,x_2)$ and $\mathcal {J}^\prime$ is the filtration $\Big\{\frac{\mathcal{J}_n+(x_2)}{(x_2)} \Big\}_{n\geq 0}$ in $A^{\prime\prime}.$
Note that $\dim A^{\prime\prime}=d-2$ and $\depth G_{\mathcal{J}}(A^\prime)\geq 1$ implies $e_{d-1}(\mathcal{J}^\prime)=e_{d-1}(\mathcal{J})$ by Singh's formula. Further,  $ e_{d-1}(\mathcal{J})=e_{d-1}(\mathcal{I})$ which is given to be zero.  By \cite[1.5]{Put2}, we obtain
\begin{equation}\label{e_d odd}
0\leq \sum_{n\geq 0} \ell\left (\frac{\widetilde{{\mathcal J}^\prime_{n+1}}}{{\mathcal J}^\prime_{n+1}}\right )=(-1)^{d-1}\big( e_{d-1}(\mathcal{J}^\prime)-{e}_{d-1}( \wt{{\mathcal J}^\prime})\big)\\
=(-1)^d {e}_{d-1}( \wt{{\mathcal J}^\prime}).
\end{equation}  
On the other hand,  since $\depth G_{ \wt{{\mathcal J}^\prime}}(A^{\prime\prime})\geq d-3$, 
by \cite[Proposition 4.4]{bl},  we have
for all $n\geq -1$,
\begin{equation} \label{23}
P_{\widetilde{{\mathcal J}^\prime}}(n)-H_{\widetilde{{{\mathcal J}^\prime}}} (n)=(-1)^{d-2}\ell((H^{d-2}_{\mathcal R_+}(\mathcal R(\widetilde{{ {\mathcal J}^\prime}})))_{n+1}). 
\end{equation}
Now, taking the sum for $m\gg 0$ on both sides of the above equation, we get
\begin{align}
(-1)^{d-2}\sum_{n=0}^m \ell((H^{d-2}_{\mathcal R_+}(\mathcal R(\widetilde{{ {\mathcal J}^\prime}})))_{n+1})&= \sum_{n=0}^m P_{\widetilde{{\mathcal J}^\prime}}(n)-\sum_{n=0}^m H_{\widetilde{{{\mathcal J}^\prime}}} (n)\nonumber\\
\begin{split}
&= \sum_{n=0}^m P_{\widetilde{{\mathcal J}^\prime}}(n)- H^2_{\widetilde{{{\mathcal J}^\prime}}} (m)
\\
&= \sum_{i=0}^{d-2} (-1)^i {e}_i(\wt{\mathcal{J}^\prime}){m+d-1-i \choose d-1-i}- P^2_{\widetilde{{{\mathcal J}^\prime}}} (m)\\
&=(-1)^{d}{e}_{d-1}(\wt{{\mathcal J}^\prime}).
\end{split}
\end{align}
This gives  $${e}_{d-1}(\wt{{\mathcal J}^\prime})\geq 0.$$ 
If $d$ is odd, then using \eqref{e_d odd}, we get ${e}_{d-1}(\wt{{\mathcal J}^\prime})=0.$ Now we show the same holds even if $d$ is even. We may assume that $d\geq 6$.  
Since $A^{\prime\prime}$ is a $(d-2)$-dimensional Cohen-Macaulay local ring,  we have $$\ell((H^{d-2}_{\mathcal R_+}(\mathcal R(\widetilde{{ {\mathcal J}^\prime}})))_n)\leq \ell((H^{d-2}_{\mathcal R_+}(\mathcal R(\widetilde{{ {\mathcal J}^\prime}})))_{n-1})$$ for all $n\in \mathbb Z$ by \cite[Lemma 4.7]{bl}. Now in equation \eqref{23}, we substitute $n=-1$ to get $$\ell((H^{d-2}_{\mathcal R_+}(\mathcal R(\widetilde{{ {\mathcal J}^\prime}})))_0)={e}_{d-2}(\wt{ {\mathcal J}^\prime})=e_{d-2}({\mathcal J}^\prime)=e_{d-2}(\mathcal J)=e_{d-2}(\mathcal{I})=0.$$
Thus, 
\begin{equation}
{e}_{d-1}(\wt{ {\mathcal J}^\prime})=\sum_{n=0}^m \ell((H^{d-2}_{\mathcal R_+}(\mathcal R(\widetilde{{ {\mathcal J}^\prime}})))_{n+1})\leq\sum_{n=0}^ m  \ell((H^{d-2}_{\mathcal R_+}(\mathcal R(\widetilde{{ {\mathcal J}^\prime}})))_0)= 0.\nonumber
\end{equation}
 Hence ${e}_{d-1}(\wt{{\mathcal J}^\prime})= 0.$ Now, using 
 \eqref{e_d odd}, we get $\wt{\mathcal{J}_{n+1}^\prime}=\mathcal{J}_{n+1}^\prime$ for $n\geq 0$. Therefore, $\depth G_{ {\mathcal J}^\prime}(A^{\prime\prime})=\depth G_{\widetilde{ {\mathcal J}^\prime}}(A^{\prime\prime})\geq d-3$. Hence by Sally's Machine \cite[Proposition 1.7]{rv}, we get $\depth G_\mathcal{J}(A^\prime)\geq d-2.$
\end{proof}
\begin{theorem}\label{signature of ed}
    Let $(A,\m)$ be a $d\geq 4$-dimensional Cohen-Macaulay local ring, $I$ an $\m$-primary ideal and $\mathcal{I}=\{I_n\}$ an $I$-admissible filtration. Suppose $e_3(\mathcal{I})=e_2(\mathcal{I})(e_2(\mathcal{I})-1)$ and $e_k(\mathcal{I})=0$ for $4\leq k \leq d-1$. Then \begin{equation}\label{e-d}
        (-1)^de_d(\mathcal{I})\geq 0.
    \end{equation}
    Furthermore, when equality holds in \eqref{e-d}, then $\depth {G}_{{\mathcal{F}}}(A)\geq d-1$. 
\end{theorem}
\begin{proof}
    For convenience of notation, we write $\mathcal{F}=\{ \wt{I_n}\}$. Let $x_1\in I$ be a superficial element for $I$. Then $x_1$ is also a superficial element for the filtration $\mathcal F$. Suppose $J=(x_1,\ldots,x_d)$ is a minimal reduction of $\mathcal{I}$. Let $A^\prime=A/(x_1)$ and ${\mathcal F}^\prime=\Big\{{\mathcal F}^\prime_n=\frac{\widetilde{I_n}+(x_1)}{(x_1)}\Big\}_{n\geq 0}$. Since $\depth G_\mathcal{F}(A)\geq 1$, hence by Singh's formula, 
    we get $e_d(\mathcal{I})=e_d(\mathcal{F})=e_d(\mathcal{F}^\prime).$
 By \cite[1.5]{Put2}, 
\begin{equation}\label{e5}
(-1)^{d}\Big( e_d(\mathcal{I})-e_d( \wt{\mathcal {F}^\prime)}\Big)=
 \sum_{n\geq 0} \ell\left (\frac{\widetilde{{\mathcal F}^\prime_{n+1}}}{{\mathcal F}^\prime_{n+1}}\right )\geq 0.
\end{equation}
By Lemma \ref{e_d-1 in d dimen}, $\depth G_{\wt{\mathcal{F}^\prime}}(A^\prime)\geq d-2$. Hence using \cite[Proposition 4.4]{bl},  we have
for all $n\geq -1$,
\begin{equation} \label{diff in 5 dim}
P_{\widetilde{{\mathcal F}^\prime}}(n)-H_{\widetilde{{{\mathcal F}^\prime}}} (n)=(-1)^{d-1}\ell((H^{d-1}_{\mathcal R_+}(\mathcal R(\widetilde{{ {\mathcal F}^\prime}})))_{n+1}). 
\end{equation}
Now, taking the sum for $m\gg 0$ on both sides of the above equation, we get
\begin{align}\label{25}
(-1)^{d-1}\sum_{n=0}^m \ell((H^{d-1}_{\mathcal R_+}(\mathcal R(\widetilde{{ {\mathcal F}^\prime}})))_{n+1})&= \sum_{n=0}^m P_{\widetilde{{\mathcal F}^\prime}}(n)-\sum_{n=0}^m H_{\widetilde{{{\mathcal F}^\prime}}} (n)\nonumber\\
\begin{split}
&= \sum_{n=0}^m P_{\widetilde{{\mathcal F}^\prime}}(n)- H^2_{\widetilde{{{\mathcal F}^\prime}}} (m)
\\
&= \sum_{i=0}^{d-1} (-1)^i e_i(\wt{\mathcal{F}^\prime}){m+d-i \choose d-i}- P^2_{\widetilde{{{\mathcal F}^\prime}}} (m)\\
&=(-1)^{d+1}e_d(\wt{\mathcal {F}^\prime}).
\end{split}
\end{align}
This gives that $e_d(\wt{\mathcal {F}^\prime})\geq 0.$  Since by \eqref{e5}, $e_d(\wt{\mathcal{F}^\prime}) \leq e_d(\mathcal{I})$, therefore $e_d(\mathcal{I})\geq 0$ if $d$ is even.
Suppose $d$ is odd. Then since $A^\prime$ is a $(d-1)$-dimensional Cohen-Macaulay local ring,  we have $$\ell((H^{d-1}_{\mathcal R_+}(\mathcal R(\widetilde{{ {\mathcal F}^\prime}})))_n)\leq \ell((H^{d-1}_{\mathcal R_+}(\mathcal R(\widetilde{{ {\mathcal F}^\prime}})))_{n-1})$$ for all $n\in \mathbb Z$ by \cite[Lemma 4.7]{bl}. Now in equation \eqref{diff in 5 dim}, we substitute $n=-1$ to get $$\ell((H^{d-1}_{\mathcal R_+}(\mathcal R(\widetilde{{ {\mathcal F}^\prime}})))_0)=e_{d-1}(\wt {\mathcal {F}^\prime})=e_{d-1}({\mathcal F}^\prime)=e_{d-1}(\mathcal F)=e_{d-1}(\mathcal{I}).$$
 Thus,
\begin{equation}
e_d( \wt{\mathcal {F}^\prime})=\sum_{n=0}^m \ell((H^{d-1}_{\mathcal R_+}(\mathcal R(\widetilde{{ {\mathcal F}^\prime}})))_{n+1})\leq\sum_{n=0}^m  \ell((H^{d-1}_{\mathcal R_+}(\mathcal R(\widetilde{{ {\mathcal F}^\prime}})))_0)=0\nonumber
\end{equation}

 Since it is given that $e_{d-1}(\mathcal{I})=0$, we get $e_d(\wt{\mathcal{F}^\prime})= 0.$ Hence by \eqref{e5}, we get $e_d(\mathcal{I})\leq e_d(\wt{\mathcal{F}^\prime})= 0.$
Now, when $e_d(\mathcal{I})=0$, then by \eqref{e5}, we get $\wt{\mathcal{F}_{n+1}^\prime}=\mathcal{F}_{n+1}^\prime$ for $n \geq 0.$ So, by Lemma \ref{e_d-1 in d dimen}, $\depth G_{\mathcal{F}^\prime}(A^\prime)= \depth G_{\wt{\mathcal{F}^\prime}}(A^\prime)\geq d-2$. 
Therefore by Sally's machine \cite[Lemma 1.4]{rv}, we get $\depth G_\mathcal{F}(A)\geq d-1.$
\end{proof}
By the standard technique of going modulo a superficial sequence of appropriate length and using the previous theorem, we get the following corollary. 
\begin{corollary}\label{cnseq of eqa of e3}
    Let $(A,\m)$ be a Cohen-Macaulay local ring of dimension $d\geq 4.$ Suppose $e_3(\mathcal{F})=e_2(\mathcal{F})(e_2(\mathcal{F})-1)$, then $e_4(\mathcal{F})\geq 0.$ Furthermore, if $e_k(\mathcal{F})=0$ for $4\leq k\leq i-1\leq d-1$, then $(-1)^ie_i(\mathcal{F})\geq 0.$
\end{corollary}

\begin{example}
 Let $A= k[[x,y,z,w]]$ and $I=(x^2-y^2, y^2-z^2, z^2-w^2, xy,yz,xz)$. Then, the Hilbert series of $I$ is 
    $$
    H(I,t)=\frac{10+12t^2-8t^3+2t^4}{(1-t)^4}.
    $$
    Here $e_0(I)=16,e_1(I)=8,e_2(I)=0,e_3(I)=0,e_4(I)=2$. Hence $e_3(I)=e_2(I)(e_2(I)-1)$ and $e_4(I)> 0$. 
    \end{example}

Now we state our results for integrally closed ideals.
\begin{theorem}\label{sig ed inte}
    Let $(A,\m)$ be a $d\geq4$ dimensional Cohen-Macaulay local ring and $I$ an $\m$- primary integrally closed ideal. If $e_3(I)=(e_2(I)-e_1(I)+e_0(I)-\ell (A/I))e_2(I)$ and $e_k(I)=0$ for $4\leq k\leq d-1$ then $(-1)^de_d(I)\geq0$. Furthermore, when equality occurs, then $\depth G_\mathcal{F}(A)\geq d-1$, where $\mathcal{F}=\{\wt{I^n}\}_{n\geq0}.$
     \end{theorem}
     \begin{proof}
     It follows exactly as the proof of Theorem \ref{signature of ed} using the next lemma which is analogous to Lemma \ref{e_d-1 in d dimen}. 
     \end{proof}
     Though the proof of the next lemma is similar to that of Lemma \ref{e_d-1 in d dimen} but for the base case it requires a bound on the reduction number $\wt{r}_J(\mathcal{I})$ in dimension two given by authors in \cite[Lemma 3.8]{sy}. 
\begin{lemma}\label{e_d-1 for integrally closed}  
    Let $(A,\m)$ be a $d\geq 4$-dimensional Cohen-Macaulay local ring, $I$ an $\m$-primary integrally closed ideal. If $e_3(I)=e_2(I)(e_2(I)-e_1(I)+e_0(I)-\ell(A/I))$ and $e_k(I)=0$ for $4\leq k \leq d-1$ then $\depth {G}_{\wt{\mathcal{F}^\prime}}(A^\prime)\geq d-2$ where $\mathcal{F}= \{\wt{I^n}\}_{n\geq 0}$ and ${\mathcal{F}^\prime}=\Big\{\frac{\wt{I^n}+(x_1)}{(x_1)}\Big\}_{n\geq 0}$ for a superficial element $x_1$ of $I.$
\end{lemma}
 \begin{proof} We use induction on $d.$ First, assume $d=4.$ Let $x_1,x_2\in I$ be a superficial sequence for $I$ such that $I/(x_1)$ and $I/(x_1,x_2)$ are  integrally closed ideals in $A/(x_1)$ and $A/(x_1,x_2)$ respectively. Suppose $J=(x_1,x_2,x_3,x_4)$ is a minimal reduction of $I$. As earlier, set $\mathcal{F}=\{\wt{I^n}\}_{n \geq 0}$, $\mathcal{F}^\prime=\Big\{\frac{\wt{I^n}+(x_1)}{(x_1)}\Big\}_{n\geq 0}$, $\mathcal{J}=\{\wt{\mathcal{F}^\prime}\}$ and $\mathcal{J}^\prime=\Big\{\frac{\mathcal{F}^\prime_n+(x_2)}{(x_2)}\Big\}_{n\geq 0}$. Since $\depth G_{\mathcal{J}}(A^\prime)\geq 1,$ hence by Singh's formula, we get $e_3(I)=e_3(\mathcal{F})=e_3(\mathcal{J})=e_3(\mathcal{J}^\prime)$. Further, by \cite[1.5]{Put2}, we have 
\begin{equation}\label{e_3 integrally}
    e_3(\wt{\mathcal{J}^\prime})-e_3(\mathcal{J}^\prime)=\sum_{n\geq 0}\ell\Big(\frac{\wt{\mathcal{J}^\prime_{n+1}}}{\mathcal{J}^\prime_{n+1}}\Big)\geq 0.
\end{equation}
Then by the similar calculations as  done in the proof of Theorem \ref{e3_m-primary} for obtaining the bound in \eqref{e3tilde}, we get, 
\begin{align*}  e_3(I)=e_3(\mathcal{F})=e_3(\mathcal{J}^\prime)\leq e_3(\wt{\mathcal{J}^\prime})\leq (\wt{r}(\mathcal{J}^\prime)-2)e_2(I).  
\end{align*}
For the filtration $\mathcal{J}^\prime$ in $A/(x_1,x_2),$ the ideal $\mathcal{J}^\prime_1=I/(x_1,x_2)$ is integrally closed.  Hence, by \cite[Lemma 3.8]{sy} we get,
\begin{align*}
 e_3(I)=e_3(\mathcal{J}^\prime)\leq e_3(\wt{\mathcal{J}^\prime})&\leq e_2(\mathcal{J}^\prime)(e_2(\mathcal{J}^\prime)-e_1(\mathcal{J}^\prime)+e_0(\mathcal{J}^\prime)-\ell(A/I))\\
&= e_2(I)(e_2(I)-e_1(I)+e_0(I)-\ell(A/I)).   
\end{align*}
 Since it is given that $e_3(I)=e_2(I)(e_2(I)-e_1(I)+e_0(I)-\ell(A/I))$. Therefore from \eqref{e_3 integrally}, we have $\wt{\mathcal{J}^\prime_{n+1}}=\mathcal{J}^\prime_{n+1}$ which gives that, $\depth G_{\mathcal{J}^\prime}(A^{\prime\prime})\geq 1$. Hence by Sally's Machine, we get $\depth G_{\wt{\mathcal{F}^\prime}}(A^\prime)\geq 2.$

Now assume $\dim A=d$ and the result holds for $d-1$. Suppose $J=(x_1,\ldots,x_d)$ is a minimal reduction of $I$. By induction hypothesis, \begin{equation}
    \depth G_{\wt{\mathcal{J}^\prime}}(A^{\prime\prime})\geq d-3
\end{equation}
where $\mathcal{J}=\wt{\mathcal{F}^\prime}$ and $\mathcal{J}^\prime=\Big\{\frac{\mathcal{J}_n+(x_2)}{(x_2)}\Big\}_{n \geq 0}$. By the similar arguments as in Lemma 
 \ref{e_d-1 in d dimen}, we conclude that $\depth G_{\wt{\mathcal{F}^\prime}}(A^\prime)\geq d-2.$\end{proof}
We also obtain the following corollary. 
\begin{corollary}\label{cnseq of eqa of e3 integrally closed}
    Let $(A,\m)$ be a Cohen-Macaulay local ring of dimension $d\geq 4$ and $I$ an $\m$-primary integrally closed ideal. Suppose $e_3(I)=e_2(I)(e_2(I)-e_1(I)+e_0(I)-\ell(A/I))$, then $e_4(I)\geq 0.$ Furthermore, if $e_k(I)=0$ for $4\leq k\leq i-1\leq d-1$, then $(-1)^ie_i(I)\geq 0.$
\end{corollary}
 \begin{proof}
Let $x_1,\ldots,x_{d-i}$ be a superficial sequence for $I$ such that $I/(x_1,\ldots,x_{d-i})$ remains integrally closed in $A/(x_1,\ldots,x_{d-i}).$ Now it follows from Theorem \ref{sig ed inte}.
\end{proof}

We refer to Example \ref{eg1} where all the conditions for the above corollary are satisfied. Hence $\depth G_\mathcal{F}(A)\geq d-1 $, with $\mathcal{F}=\{\wt{m^n}\}_{n\geq 0.}$


\section{application}\label{section-5}
In this section, we give some interesting applications of our theorem. In \cite[Theorem 4.3]{nq}, Cleto and Queiroz proved that when $\wt{I^{r_J(I)}}=I^{r_J(I)}$ and the Ratliff-Rush filtration of $I$ behaves well mod superficial sequence of length $d-2$, then $\reg G_I(A)=r_J(I)$ and asked a question \cite[Question 4.4]{nq} whether one can  get the same conclusion assuming only  the equality $\wt{I^{r_J(I)}}=I^{r_J(I)}$. In  Theorem \ref{cleto3}, we give a partial answer of their question. First we prove a bound for $\wt{r}_J(\mathcal{I})$ and the stability index $\rho(I)$ which partially answers the questions raised in \cite{rs}.
\begin{theorem}\label{Theorem5.1}
Let $(A, \mathfrak{m})$ be a Cohen-Macaulay local ring of dimension $d \geq 3$, $I$ an $\mathfrak{m}$-primary ideal and $\mathcal{I}$ an $I$-admissible filtration. If either of the following two conditions holds:
 \renewcommand{\labelenumi}{(\roman{enumi})}
\begin{enumerate}
    \item $e_3(\mathcal{I}) = e_2(\mathcal{I})(e_2(\mathcal{I}) - 1)$,
    \item $\mathcal{I}=\{I^n\}_{n\geq 0}$, $I$ is integrally closed and $e_3(I) = e_2(I)(e_2(I) - e_1(I) + e_0(I) + \ell(A/I))$,
\end{enumerate}
and $e_k(\mathcal{I}) = 0$ for $4 \leq k \leq d$.
Then the following holds:
 \renewcommand{\labelenumi}{(\roman{enumi})}
 \begin{enumerate}
        \item \label{Part1}$\wt{r}_J(\mathcal{I})\leq r_J(\mathcal{I}).$
        \item \label{Part2} If $\wt{I_{r}}=I_{r}$ for some $r \geq r_J(\mathcal{I})$, then $\rho(\mathcal{I})\leq r.$ Furthermore, 
 \begin{eqnarray*}
            \rho(\mathcal{I})\leq \begin{cases}
                r_J(\mathcal{I}) &\text{ if } \wt{I_{r_J(\mathcal{I})}}=I_{r_J(I)}\\
                r_J(\mathcal{I})+(-1)^{d+1}(e_{d+1}(\mathcal{I})-e_{d+1}(\wt{\mathcal{I}})) &\text{ otherwise. } 
            \end{cases}
        \end{eqnarray*}
        
    \end{enumerate}
    \end{theorem}
 \begin{proof}
  \renewcommand{\labelenumi}{(\roman{enumi})}
     \begin{enumerate}
  \item Let $\underline{x}=x_1,\ldots,x_{d-2}\in I$ be a superficial sequence of $I$. By Theorems \ref{signature of ed} and \ref{sig ed inte}, $\depth G_\mathcal{F}(A)\geq d-1$ which gives $\wt{r}_J(\mathcal{I})=\wt{r}_{J/(\underline{x})}(\mathcal{I}/(\underline{x}))$, see \cite[Lemma 2.14]{tm}. Note that the proof of \cite[Lemma 2.14]{tm} works for filtration as well. By \cite[Proposition 2.1]{ms}, we have $\wt{r}_{J/(\underline{x})}(\mathcal{I}/(\underline{x}))\leq r_{J/(\underline{x})}(\mathcal{I}/(\underline{x}))\leq r_J(\mathcal{I}).$ Therefore, $\wt{r}_J(\mathcal{I})\leq r_J(\mathcal{I}).$
  \item Suppose $\wt{I_r}=I_r$ for some $r\geq r_J(\mathcal{I})$. It is enough to show that $\wt{I_{r+1}}=I_{r+1}.$ Since $\wt{r}_J(\mathcal{I})\leq r_J(\mathcal{I})$. Then $I_{r+1}=JI_r=J\wt{I_r}=\wt{I_{r+1}}.$ Therefore $\rho (\mathcal{I})\leq r.$ For the next part, let $r=r_J(\mathcal{I})$.  Assume 
 $\wt{I_r}\neq I_r$, then $(-1)^{d+1}(e_{d+1}(\mathcal{I})-{e}_{d+1}(\wt{\mathcal{I}}))=\mathop\sum\limits_{n\geq 0}\ell(\wt{I_{n+1}}/I_{n+1})>0.$ If $\rho(\mathcal{I})\leq r_J(\mathcal{I}),$ then we are done. So, we may assume that   $\rho(\mathcal{I})> r_J(\mathcal{I}).$ Then, for all $r\leq n < \rho(\mathcal{I}),$
 $$\wt{I_n}\neq I_n.$$  To see this, suppose $\wt{I_n}=I_n$ for some $n$ with $r\leq n < \rho(\mathcal{I}),$ then by above arguments, we have $\wt{I_t}=I_t$ for all $t\geq n$, which implies $\rho(\mathcal{I})\leq n$, a contradiction. 
 Thus $\ell(\wt{I_{n}}/I_{n})\geq 1$ for $$r\leq n < \rho(\mathcal{I})$$ which gives 
  \begin{eqnarray*}
      \rho(\mathcal{I})-r_J(\mathcal{I}) &\leq& \sum_{n=r_J(\mathcal{I})-1}^{\rho(\mathcal{I})-2}\ell(\wt{I_{n+1}}/I_{n+1})\\
&\leq& \sum_{n=0}^{\rho(\mathcal{I})-2}\ell(\wt{I_{n+1}}/I_{n+1})\\
&=& (-1)^{d+1}(e_{d+1}(\mathcal{I})-{e}_{d+1}(\wt{\mathcal{I}})).
\end{eqnarray*}
\end{enumerate}\end{proof}
 In particular, If $d$ is odd in the above result, then either by Theorem \ref{signature of ed} or by Theorem \ref{sig ed inte}, $\depth G_\mathcal{F}(A)\geq d-1$ which implies  ${e}_{d+1}(\wt{\mathcal{I}})=\displaystyle{\sum_{n\geq d}}\binom{n}{d}\ell(\wt{I_{n+1}}/J\wt{I_n})\geq 0$, see \cite[Theorem 2.5]{rv} for the formula. Therefore, we get $\rho (\mathcal{I})\leq r_J(\mathcal{I})+e_{d+1}(\mathcal{I}).$\\
In the following theorem, we give an affirmative answer of  \cite[Question 4.4]{nq}  under the same hypothesis as in Theorem \ref{Theorem5.1} in case of $I$-adic filtration. For the rest of this section, let $\mathcal{I}=\{I^n\}_{n\geq 0}.$ 
\begin{theorem}\label{cleto3}
Let the hypothesis be same as in Theorem \ref{Theorem5.1}. Suppose $\wt{I^{r_J(I)}}=I^{r_J(I)}$. Then  $$\reg G_I(A)=r_J(I).$$ 
\end{theorem}
\begin{proof}
    Set $r=r_J(I)$ and $A_j=A/(x_1,\ldots,x_{j-1})$  for $2\leq j \leq d-1$ with $A_1=A$ where $x_1,\ldots,x_{d-2}$ is a superficial sequence for $I.$
    Then from Theorem \ref{signature of ed} or from Theorem \ref{sig ed inte}, we have $\depth G_\mathcal{F}(A)\geq d-1$. Consequently, 
    by Sally's machine, $\depth G_{\mathcal{F}A_j}(A_j)\geq d-j\geq 1$ for $2 \leq j \leq d-1.$ Hence $\wt{(\mathcal{F}A_j)_n}=(\mathcal{F}A_j)_n$ for $2\leq j\leq d-1$ and for all $n\geq 0$. In particular $\wt{(\mathcal{F}A_j)_r}=(\mathcal{F}A_j)_r$ for $2 \leq j \leq d-1$. Since it is given that $\wt{I^r}=I^r$, we get   
    $$ \wt{I^rA_j}=I^rA_j$$ for $2\leq j \leq d-1$.
It follows from \cite[Theorem 4.3]{nq} that $\reg R(I)=\reg G_I(A)=r_J(I)$.
\end{proof}
\begin{theorem}\label{thm-red}
    Let the hypothesis be the same as in Theorem \ref{Theorem5.1}. Then Rossi's bound holds, i.e, $r_J(I)\leq e_1(I)-e_0(I)+\ell(A/I)+1.$
\end{theorem}
\begin{proof}
By Theorem \ref{signature of ed} or by Theorem \ref{sig ed inte}, we have $\depth G_{\wt{I}}(A)\geq d-1$  so using \cite[Theorem 1.3]{r}, we have $r_J(I)\leq \displaystyle{\sum_{n\geq 0}}\ell(\wt{I^{n+1}}/J\wt{I^n})-e_0(I)+\ell(A/I)+1=e_1(I)-e_0(I)+\ell(A/I)+1$.    
\end{proof}

\end{document}